\documentclass[11pt]{amsart}
\usepackage{a4wide}
\usepackage{kotex}
\usepackage{amsmath}
\usepackage{amsfonts}
\usepackage{amscd}

\usepackage{tikz}

\usepackage{amssymb}
\usepackage{graphicx}
\usepackage{dsfont}
\usepackage{color}
\usepackage{bbm}
\usepackage[colorlinks]{hyperref}
\hypersetup{citecolor=blue,}
\hypersetup{urlcolor=black,}
\usepackage{url}

\usepackage{xurl}

\usepackage{caption}
\usepackage{subcaption}
\usepackage{sseq}
\usepackage{tikz-cd}
\usetikzlibrary{arrows.meta}
\usepackage[abs]{overpic} 
\usepackage{bm}

\theoremstyle{plain}

\newtheorem{theorem}{Theorem}[section]
\newtheorem{lemma}[theorem]{Lemma}

\newtheorem{proposition}[theorem]{Proposition}
\newtheorem{corollary}[theorem]{Corollary}

\newtheorem*{question*}{Question}

\theoremstyle{definition}
\newtheorem*{definition*}{Definition}
\newtheorem{definition}[theorem]{Definition}

\newtheorem*{example*}{Example}
\newtheorem{example}[theorem]{Example}
\newtheorem*{observation*}{Observation}

\newtheorem*{Goal*}{Goal}

\newtheorem*{Assumption*}{Assumption}

\theoremstyle{remark}
\newtheorem*{remark*}{Remark}
\newtheorem{remark}[theorem]{Remark}

\numberwithin{equation}{section}

\newcommand{\ow}{\omega}

\newcommand{\lda}{\lambda}
\newcommand{\Lda}{\Lambda}
\newcommand{\p}{\partial}

\newcommand{\C}{{\mathbb{C}}}
\newcommand{\R}{{\mathbb{R}}}
\newcommand{\Q}{{\mathbb{Q}}}
\newcommand{\Z}{{\mathbb{Z}}}
\newcommand{\N}{{\mathbb{N}}}

\newcommand{\CP}{\C P}

\DeclareMathOperator{\Spec}{Spec}

\DeclareMathOperator{\CZ}{CZ}

\DeclareMathOperator{\CF}{CF}
\DeclareMathOperator{\SH}{SH}
\DeclareMathOperator{\HF}{HF}
\DeclareMathOperator{\Ho}{H}

\DeclareMathOperator{\QH}{QH}

\begin{document}

\title{Transfer maps in symplectic cohomology for convex symplectic domains}

\author{Myeonggi Kwon}
\address{Department of Mathematics Education, and Institute of Pure and Applied Mathematics, Jeonbuk National University, Jeonju 54896, Republic of Korea}
\email{mkwon@jbnu.ac.kr}

\author{Takahiro Oba}
\address{Department of Mathematics, The University of Osaka, Toyonaka, Osaka 560-0043, Japan}
\email{taka.oba@math.sci.osaka-u.ac.jp}

\begin{abstract}
 
    We construct transfer maps in symplectic cohomology for convex symplectic domains, assuming that the complement of a subdomain is exact. The construction uses an action filtration by Reeb periods, inspired by that of McLean–Ritter. The transfer maps yield the invariance of symplectic cohomology under subcritical handle attachment to convex symplectic domains. As an application, we obtain infinitely many exotic convex symplectic structures on certain domains via McLean’s exotic Stein balls.
    
\end{abstract}


\maketitle


\section{Introduction}

Transfer maps in Floer theory, introduced by Viterbo \cite{Vit}, have played a central role in fundamental questions of symplectic topology. For example, they serve as nontrivial obstructions to the existence of certain Lagrangian submanifolds \cite{Vit} and are essential ingredients for constructing symplectic capacities from Floer theory, as, e.g., in \cite{FHW, GH18}. 

Roughly speaking, for a symplectic manifold $V$ and a codimension zero embedding $W \hookrightarrow V$, transfer maps appear as natural homomorphisms between various flavors of Floer cohomology of $V$ and $W$.
To our knowledge, such maps have been constructed only when the ambient symplectic manifold $V$ is globally exact in the literature, and 
the purpose of this article is to provide a construction of transfer maps with favourable functorial properties in symplectic cohomology for possibly non-exact symplectic manifolds. We carry out this construction when the complement of a subdomain is exact.



Let $V$ be a convex symplectic domain, meaning that it is a compact symplectic manifold such that the symplectic form $\ow$ is exact near the boundary $\p V$, and the Liouville vector field points outward along $\p V$. Under some additional assumptions depending on the context,
the symplectic cohomology $\SH^*(V)$ is defined as a $\Z$-graded algebra over the Novikov field $\Lda_{V}$. A standard construction of $\SH^*(V)$ is briefly reviewed in Section \ref{sec: symcohomology}; we refer the reader to \cite{BeRi20} for an intensive study on the invariance of symplectic cohomology of convex symplectic domains.

Let $W$ be a codimension zero subdomain of $V$ such that the primitive $1$-form $\lda$ defined near $\p V$ extends to a neighborhood of $V \setminus \text{int} W$ and satisfies $\ow = d\lda$ there; we call it a \emph{complement-exact} subdomain. When the boundary $\p W$ is simply-connected, the Novikov field $\Lda_W$ associated with $W$ can be canonically identified with the one $\Lda_V$ of the ambient domain $V$; see Theorem \ref{thm: nonnegandW}. Under this identification, the main result of this article is to establish an algebra homomorphism from $\SH^*(V)$ to $\SH^*(W)$ that has a functorial property with respect to the restriction map $i^*\colon \QH^*(V) \rightarrow \QH^*(W)$ on the quantum cohomology rings.




\begin{theorem}\label{thm: A}
Let $W$ be a complement-exact subdomain in a convex symplectic domain $V$. Assume that $c_1(V)|_{\pi_2(V)} = 0$ and $\pi_1(\p W) = 0$. There exists a $\Z$-graded algebra homomorphism
$$
\Phi\colon \SH^*(V; \Lda_V) \rightarrow \SH^*(W; \Lda_W)
$$
under a canonical identification $\Lda_V = \Lda_W$. Moreover, it fits into the following commutative diagram with quantum cohomology rings.
\[\begin{tikzcd}
        \SH^*(V) \arrow{r}{\Phi}  & \SH^*(W)  \\
        \QH^*(V) \arrow{r}{i^*} \arrow{u}{c^*}& \QH^*(W) \arrow[swap]{u}{c^*}
    \end{tikzcd}
\]
where the vertical maps are the canonical $c^*$-maps. 
\end{theorem}

The $c^*$-map above is induced by a natural inclusion involved in the definition of symplectic cohomology; see \cite[Section 4.1]{BeRi20} and \cite[Section 5]{Rit}.

For the construction, we essentially follow McLean \cite[Section 10.2]{Mclean_extension} and Ritter \cite[Section 9]{Rit} where transfer maps are constructed in symplectic (co)homology for exact domains. An essential ingredient is to use a special type of admissible Hamiltonians that are adapted to the embedding $W \hookrightarrow V$, see Figure \ref{fig: tranfer_Ham}, and also properly interact with an action filtration on the symplectic (co)homology. In particular, a careful estimate of the action values of the generators of the corresponding Hamiltonian Floer cochain complex shows that the part $\SH^*_{\leq 0}(V)$ whose generators have non-positive action is canonically isomorphic to the symplectic cohomology $\SH^*(W)$ of the subdomain $W$. Then, the natural quotient map $\SH^*(V) \rightarrow \SH^*_{\leq 0}(V)$, modding out by the generators with positive action, produces the desired homomorphism $\SH^*(V) \rightarrow \SH^*(W)$. This was shown to be a ring homomorphism \cite[Section 10.2]{Mclean_extension} and, more generally, compatible with TQFT operations \cite[Section 9]{Rit} for exact domains.

The main technical point of the current non-exact setup is that the standard action functional, as in Remark \ref{rem: usualactionfctl}, now depends not only on loops but also on their capping disks; this makes the action value estimates delicate. To handle this, we introduce a new action functional, inspired by the work of McLean--Ritter \cite{McRi23}, which is designed to deal with non-exact convex symplectic domains. The key idea is to define the action functional so that it vanishes in the non-exact region. Consequently, only the exact part contributes to the action values, and these values can be effectively understood through Reeb periods of the contact boundary.


In Section \ref{sec: newfiltration}, we define an action functional $\mathcal{F}$ and provide the relevant action value estimates in Lemma \ref{lem: actionorbits}. In particular, we show that the non-positive part $\SH^*_{\leq 0}(V)$ with respect to the action filtration induced by $\mathcal{F}$ corresponds to the symplectic cohomology $\SH^*(W)$ of the subdomain $W$; see Theorem \ref{thm: nonnegandW}. It is also shown in Corollary \ref{cor: algebraisom} that the correspondence is compatible with the respective ring structures. The commutative diagram in Theorem \ref{thm: A} then follows from a standard argument.

There are various interesting examples of convex symplectic domains with complement-exact subdomains for potential applications. We consider negative line bundles in Example \ref{ex: neglinebdl2} and resolutions of isolated singularities in Example \ref{ex: resolution}. In addition, for a given convex domain $W$, attaching an exact cobordism along the boundary produces a larger convex domain $V$, and $W$ is complement-exact in $V$. 

In particular, as a generalization of the seminal result of Cieliebak \cite{Ci02}, we observe in Theorem \ref{thm: inv_subcrit_hdl} that transfer maps yield an isomorphism under subcritical handle attachment to convex symplectic domains. As an application, we exhibit infinitely many convex symplectic structures on certain domains, including negative line bundles, that are pairwise deformation non-equivalent. See Theorem \ref{thm: exotic_str}.

\section{Symplectic cohomology for convex symplectic domains}

\subsection{Convex symplectic domains} 
We recall the notion of convex symplectic domains following \cite{BeRi20}. 

Let $(V, \ow)$ be a compact symplectic manifold with boundary $\p V$ and $\lda$ a $1$-form defined in a neighborhood of $\p V$ such that $\ow = d \lda$; in particular, $\ow$ is exact near $\p V$. The triple $(V, \ow, \lda)$ is called a \emph{convex symplectic domain} if the Liouville vector field $X$, defined by $\iota_X \ow = \lda$, is pointing outward along $\p V$. Liouville domains are examples of convex symplectic domains, but convex symplectic domains are not necessarily exact (nor symplectically aspherical) in the interior. Note that the restriction $\alpha : =  \lda|_{\p V}$ defines a contact structure $\xi : = \ker \alpha$ on $\p V$, and the domain $V$ forms a (strong) symplectic filling of the contact manifold $(\p V, \xi)$.

\begin{example}\label{ex: neglinebdl}
A class of examples of convex symplectic domains, which are not exact, can be obtained from negative line bundles over a closed integral symplectic manifold. More precisely, let $(B, \ow)$ be an integral symplectic manifold, and 
let $\pi\colon E \rightarrow B$ be the associated line bundle with $c_1(\pi) = - [\ow]$. Then the total space of a disk bundle $\pi\colon D(E) \rightarrow B$ serves as a convex symplectic domain which is not exact. Other related interesting examples, coming from the notion of \emph{Lefschetz--Bott fibrations}, can be found in \cite{Oba22}.



\end{example}



Gluing the positive symplectization $[1, \infty) \times \p V$ to $V$ along the boundary using the Liouville flow, we can complete $(V, \ow, \lda)$ as follows:
$$
\widehat V : = V \cup_{\p V} ([1, \infty) \times \p V), \quad \widehat \ow : = \ow \cup d(r \alpha), \quad \widehat \lda = \lda \cup r\alpha
$$
where $r \in [1, \infty)$. The open symplectic manifold $(\widehat V, \widehat \ow, \widehat \lda)$ is called the \emph{completion} of $V$. 

More generally, we define a \emph{convex symplectic manifold} $(M, \ow, \lda)$ to be an open symplectic manifold $(M, \ow)$ such that 
\begin{itemize}
\item there exists an exhausting function $h\colon M \rightarrow \R$;
\item $\lda$ is a 1-form defined on $\{h(z) \geq 1\} \subset M$ such that $\ow = d\lda$ and $\lda(X_h) > 0$ where $X_h$ is the Hamiltonian vector field of $h$. (This in particular implies that the associated Liouville vector field points outward along a level set of $h$.)
\end{itemize}
In particular, the completion $(\widehat V, \widehat \ow, \widehat \lda)$ of a convex symplectic domain $V$ is a convex symplectic manifold with respect to a monotone increasing function $h\colon \widehat V \rightarrow \R$ such that $h(z) = r$, where $r \in [1, \infty)$ is the Liouville coordinate of $z \in [1, \infty) \times \p V$. The domain $V$ appears as the sublevel set $\{h \leq 1\}$.
See \cite{BeRi20} for more details.

\subsection{Complement-exact subdomains} \label{sec: comp-exact} A compact submanifold $W \subset V$ with boundary $\p W$ of codimension $0$ is called a \emph{convex symplectic subdomain}, or a \emph{subdomain} for short, if $(W, \ow|_{W})$ forms a convex symplectic domain. In this paper, we are interested in a compact submanifold $W \subset V$ for which the primitive $1$-form $\lda$ extends to a neighborhood of $V \setminus \text{int} W$ and satisfies $\ow = d\lda$ there. In this case, we say that the subdomain $W$ is \emph{complement-exact}. Note that the Liouville vector field is then well-defined in the complement, including a collar neighborhood of $\p W$. 

\begin{example}\label{ex: neglinebdl2}
Following Example \ref{ex: neglinebdl}, the total space $E$ can be seen as a (completed) convex symplectic manifold, and the total space of an associated disk bundle $\pi\colon D(E) \rightarrow B$ now serves as a complement-exact subdomain. Assume that
\begin{itemize} 
\item $[\ow] = c_1(B)$ and $[\ow]$ is primitive in $\Ho^2(B; \Z)$;
\item $B$ is simply-connected.
\end{itemize}
Then $c_1(E) = \pi^{*}(c_1(B) + c_1(\pi)) = 0$, and the homotopy exact sequence for the corresponding circle bundle $\pi: S(E) \rightarrow B$ shows that $\pi_1(S(E)) = 0$. Therefore, this example satisfies the topological assumptions in Theorem \ref{thm: A}. 

More concretely, let $B:=X_{n+1}\subset \CP^{n+1}$, $n  \geq 3$, be a smooth hypersurface of degree $n+1$, and let $\pi\colon K_B \rightarrow B$ be the canonical line bundle. By the
adjunction formula,
\[
K_B=\mathcal O_B(n+1 -(n+2))=\mathcal O_B(-1).
\]
Hence $-K_B=\mathcal O_B(1)$ is ample, so $B$ is Fano, and
\[
c_1(B)=-c_1(\pi)=c_1(\mathcal O_B(1))=H|_B \in \Ho^2(B; \Z)
\]
where $H\in \Ho^2(\CP^{n+1}; \Z)$ denotes the hyperplane class. We take $\ow$ to be the restriction to $B$ of the Fubini--Study form normalized to represent $H$. Then $[\ow] = H|_B = c_1(B)$, as assumed above. By the Lefschetz hyperplane theorem,
\[
\Ho^2(B; \Z)\cong \Z\langle H|_B\rangle
\quad\text{and}\quad
\pi_1(B)=0.
\]
Therefore $c_1(B)\in \Ho^2(B; \Z)$ is primitive and $B$ is
simply-connected.
\end{example}

\begin{example}\label{ex: handle}
    Let $W$ be a convex symplectic domain. Attaching an exact symplectic cobordism $X$ to $W$ along the boundary $\p W$ produces another convex symplectic domain $V = W \cup_{\p W} X$, and $W$ is a complement-exact subdomain of $V$. Notable examples are those obtained by the Weinstein handle attachment. See Section \ref{sec: subcrt_handle}.
\end{example}

\begin{example}\label{ex: resolution}
    Let $X$ be a $\Q$-factorial variety with a unique singularity at the origin $O \in X$ and assume that the complement $X \setminus \{O\}$ admits an exact K\"ahler form $\ow_X$. Then by \cite[Lemma A.2]{McRi23}, any resolution $\pi\colon Y \rightarrow X$ admits a K\"ahler form $\ow_Y$ such that $\pi^* \ow_X = \ow_Y$. In particular, $Y$ is a convex symplectic manifold, and a neighborhood of $\pi^{-1}(\{O\}) \subset Y$ serves as a complement-exact subdomain.
\end{example}


\subsection{Symplectic cohomology} \label{sec: symcohomology} In this section, we briefly provide a construction of symplectic cohomology for convex symplectic domains. We refer the reader to \cite{BeRi20, HoSa95, Ri14} for more details.

\subsubsection{Admissible Hamiltonians} \label{sec: admHam} Let $(V, \ow, \lda)$ be a convex symplectic domain. We assume that the first Chern class $c_1(V)$ vanishes on the second homotopy group $\pi_2(V)$.
A Hamiltonian $H\colon \widehat V \rightarrow \R$ is called \emph{admissible} if 
\begin{itemize}
\item there exists $r_0 \in [1, \infty)$ such that $H$ depends only on the radial coordinate $r$ for $r \geq r_0$, say $H(z) = h(r)$, and $h'(r) \geq 0$;
\item $h(r) = \kappa r + b$ for sufficiently large $r$, where the \emph{slope} $\kappa$ is not the period of the Reeb orbit on the contact boundary $(\p V, \alpha)$.
\end{itemize}


\begin{remark}\label{rem: whyc_1zero}  \
\begin{enumerate}
    \item The condition $c_1(V)|_{\pi_2(V)} = 0$ is to equip the symplectic cohomology with a $\Z$-grading by the Conley--Zehnder index. More generally, one can also work with the so-called \emph{weakly monotone} symplectic manifolds; see \cite{HoSa95}.
    \item As is fairly standard in Floer theory \cite{FHS, SaZe92}, we need to use generic time-dependent perturbations of admissible Hamiltonians and almost complex structures
    so that 1-periodic Hamiltonian orbits are nondegenerate and the relevant moduli spaces of Floer solutions are smooth. 
    By abuse of notation, we will continue to denote the perturbed Hamiltonians
    by $H$ throughout this paper.
\end{enumerate}

\end{remark}

\subsubsection{Novikov field} Define a group
$$
\Gamma_V = \pi_2(V) / \ker I_{\ow}
$$
where $I_{\ow}\colon \pi_2(V) \rightarrow \R$, $I_{\ow}(\gamma) := \int_{S^2} \gamma^* \ow$.
The \emph{Novikov field} $\Lda_V$ associated with $V$ is defined by
$$
\Lda_V = \left\{\sum_{j=0}^{\infty} n_j \gamma_j \;\bigg|\;  n_j \in \Z_2, \; \gamma_j \in \Gamma_V, \; \lim_{j \rightarrow \infty} I_{\ow}(\gamma_j) =  \infty  \right\}.
$$
By \cite[Theorem 4.1]{HoSa95}, $\Lda_V$ is a field of characteristic two.

\subsubsection{Cochain complex} Consider a pair $(x, v)$ of a contractible free loop $x\colon S^1 \rightarrow V$ and a smooth capping disk $v\colon D^2 \rightarrow V$ of $x$. Define an equivalence relation $\sim$ on the set of such pairs by 
$$
(x_1, v_1) \sim (x_2, v_2) \Leftrightarrow x_1 = x_2 \;   \text{ and $v_1 \# \overline{v_2} \in \ker I_{\ow}$ where $v_1 \# \overline{v_2}$ is the glued sphere.}
$$
As noted in \cite[Section 5]{HoSa95}, the set of equivalence classes, each still denoted by $(x, v)$, serves as the covering space $\widetilde{\mathcal{L}_0 V}$ of the space of free contractible loops  $\mathcal{L}_0 V$. 
\begin{remark}\label{rem: usualactionfctl}
For an admissible Hamiltonian $H$, a standard action functional $\mathcal{A}_H \colon \widetilde{\mathcal{L}_0 V} \rightarrow \R$ is given by
$$
\mathcal{A}_H(x, v) =  -\int_{D^2} v^*\ow + \int_{S^1} H(x(t)) dt.
$$
The class $(x, v)$ is a critical point of $\mathcal{A}_H$ if $x$ is a $1$-periodic Hamiltonian orbit of $H$. Here, we use the convention $\ow(\cdot, X_H)  = d H$ for the definition of the Hamiltonian vector field $X_H$.
\end{remark}

With each pair $(x, v) \in \widetilde{\mathcal{L}_0 V}$, we can associate the Conley--Zehnder index $\CZ(x, v) \in \Z$, as in \cite[Section 5]{HoSa95}; see also Remark \ref{rem: whyc_1zero}. Since we have assumed that $c_1(V)|_{\pi_2(V)} =0$, the index $\CZ(x, v)$ does not actually depend on the choice of capping disks $v$; we abbreviate it to $\CZ(x)$. For each $k \in \Z$, denote by $\widetilde{\mathcal{P}}_k(H)$ the set of pairs $(x, v)$ with $n- \CZ(x) = k$. We define a chain group $\CF^*(H)$ by
$$
\CF^k(H) = \left\{ \sum_{j=0}^{\infty} n_jc_j \;\bigg|\;  n_j \in \Z_2, \; c_j \in \widetilde{\mathcal{P}}_k(H),\; \lim_{j \rightarrow \infty}\mathcal{A}_H(c_j) \mathrel{=}  -\infty \right\}
$$
which is a vector space over $\Z_2$ generated by $\widetilde{\mathcal{P}}_k(H)$.
As discussed in \cite[Section 5]{HoSa95}, we can identify $\CF^k(H)$ with the vector space over the Novikov field $\Lda_V$ given by
$$
\CF^k(H) = \bigoplus_{x \in \mathcal{P}_k(H)} \Lda_V \langle x \rangle
$$
where $\mathcal{P}_k(H)$ denotes the set of 1-orbits $x$ with $n- \CZ(x) =  k$. 

Take an \emph{admissible} time-dependent almost complex structure $J = J_t$ on $\widehat V$, which means that $J$ is compatible with $\widehat \ow$ and is cylindrical at the end, i.e. $J^*\widehat \lda = dr$. For two generators $(x_{\pm}, v_{\pm}) \in \widetilde{\mathcal{P}}_*(H)$, consider the moduli space $\mathcal{M}((x_-, v_-) , (x_+, v_+); H, J)$ consisting of Floer solutions $u\colon \R \times S^1 \rightarrow \widehat V$, with $\displaystyle \lim_{s \to \pm \infty} u(s, \cdot) = x_{\pm}(\cdot)$, of the Floer equation 
\[
\p_s u + J_t (\p_t u - X_H(u)) = 0
\]
such that $(x_+, v_+)  = (x_+, v_- \#u) \in \widetilde{\mathcal{L}_0 V}$ up to the $\R$-shift, i.e. $u \sim u(\cdot + \text{const}, \cdot)$. As in \cite{FHS, SaZe92}, for generic $J$, the moduli space $\mathcal{M}((x_-, v_-) , (x_+, v_+); H, J)$ is a smooth manifold of dimension $\CZ(x_+)  - \CZ(x_-) - 1$. We define the differential map $d \colon \CF^*(H) \rightarrow \CF^{*+1}(H)$
by counting the elements in $\mathcal{M}((x_-, v_-) , (x_+, v_+); H, J)$ with $\CZ(x_+)  - \CZ(x_-) - 1 = 0$, that is,
\[
d(x_+, v_+) = \sum_{\CZ(x_-) = \CZ(x_+) - 1} (\#_{\Z_2} \mathcal{M}((x_-, v_-) , (x_+, v_+); H, J)) (x_-, v_-).
\]

\subsubsection{Symplectic cohomology} The cohomology group $\HF^*(H) = \HF^*(H, J; \Lda_V)$ of the complex $(\CF^*(H), d)$ is called the \emph{Hamiltonian Floer cohomology} of $H$. For two admissible Hamiltonians $H_{\pm}$ with $H_+ \leq H_-$, there is a canonical homomorphism, called a \emph{continuation map} $\HF^*(H_+) \rightarrow \HF^*(H_-)$; see \cite[Section 5]{HoSa95} and \cite[Section 2.8]{Ri14}. This gives rise to a direct system by increasing the slope $\kappa$ of the Hamiltonians. The \emph{symplectic cohomology} $\SH^*(V) = \SH^*(V; \Lda)$ of the domain $V$ is defined as the direct limit
$$
\SH^*(V) = \SH^*(V; \Lda) = \varinjlim_{\kappa} \HF^*(H)
$$ 
which is a $\Z$-graded vector space over the Novikov field $\Lda_V$.



\subsubsection{Ring structure} The symplectic cohomology $\SH^*(V)$ admits a ring structure via the pair-of-pants product. 
Let $\mathcal{S}$ be the Riemann sphere with two positive punctures and one negative puncture; this means that $\mathcal{S}$ admits parametrizations by $[0, \infty) \times S^1$ and $(-\infty, 0] \times S^1$, respectively, near the punctures. Given three admissible Hamiltonians $H_1, H_2, H_3$ on $\widehat V$, we take an $\mathcal{S}$-parametrized Hamiltonian $H_{\mathcal{S}}$ which coincides with $H_1, H_2$ near the positive punctures and with $H_3$ near the negative puncture. We likewise take an $\mathcal{S}$-parametrized admissible almost complex structure $J_{\mathcal{S}}$. We define a product
\[
\HF^k(H_1) \otimes \HF^{\ell}(H_2) \rightarrow \HF^{k+\ell}(H_3), \quad x_1 \otimes x_2 \mapsto x_3
\]
by counting the solutions $u\colon \mathcal{S} \rightarrow \widehat V$ of the Floer equation
\begin{equation} \label{eq: Floereq_ring}
   (du - X_{H_{\mathcal{S}}} \otimes \beta)^{0, 1} = 0 
\end{equation} 
which converge to $x_1, x_2$ at the positive punctures and converge to $x_3$ at the negative puncture. Here, $\beta \in \Omega^1(\mathcal{S})$ is a one-form that agrees with $dt$ near the punctures (where $t$ denotes the $S^1$-coordinates of the parametrizations). As is well-known, the product is compatible with the continuation maps and hence induces a product on the symplectic cohomology as
\[
\SH^k(V) \otimes \SH^{\ell}(V) \rightarrow \SH^{k+\ell}(V).
\]
With this product, the symplectic cohomology $\SH^*(V)$ is now a unital $\Z$-graded algebra over the Novikov field $\Lda_V$. For more details on the construction of the ring structure, we refer the reader to \cite[Section 6]{Rit} and \cite[Section 4.7]{BeRi20}.

\subsubsection{Invariance of symplectic cohomology} \label{sec: inv_SH_deformation}

As in \cite[Definition 2.3]{BeRi20}, a family $(M, \ow_s, \lda_s)$ with $s \in [0, 1]$ of convex symplectic manifolds is called a \emph{deformation} if there exists an exhausting function $H \colon [0, 1] \times M \rightarrow \R$ such that $H(s, \cdot) = h_s(\cdot)$, where $h_s \colon M \rightarrow \R$ is the exhausting function associated with $(M, \ow_s, \lda_s)$. 

The following definition is extracted from \cite[Theorem 1.3]{BeRi20}; in particular, we refer the reader to \cite[Lemma 3.1 and Definition 3.3]{BeRi20} for the role of the condition below regarding the relative cohomology class $[\widehat{\ow}_s, \widehat{\lda}_s] \in \Ho_c^2(\widehat{V})$. 

\begin{definition}\label{def: deform_equiv}
    Let $V_0 = (V, \ow_0, \lda_0)$ and $V_1 = (V, \ow_1, \lda_1)$ be convex symplectic domains. $V_0$ is \emph{deformation equivalent} to $V_1$ if there exists a deformation $(\widehat{V}_s, \widehat{\ow}_s, \widehat{\lda}_s)$ such that the cohomology class $[\widehat{\ow}_s, \widehat{\lda}_s] \in \Ho_c^2(\widehat{V})$ is constant for $s \in [0, 1]$.
\end{definition}


By Benedetti--Ritter \cite[Theorem 1.3]{BeRi20}, symplectic cohomology is invariant under deformation equivalences of convex symplectic domains:

\begin{theorem}[Benedetti--Ritter]\label{thm: inv_SH_BR}
    If $V_0 = (V, \ow_0, \lda_0)$ and $V_1 = (V, \ow_1, \lda_1)$ are deformation equivalent, then $\widehat{V}_0$ and $\widehat{V}_1$ are isomorphic as convex symplectic manifolds, i.e., there exists a symplectomorphism $\varphi \colon \widehat{V}_0 \rightarrow \widehat{V}_1$ such that $\varphi^* \widehat{\lda}_1 = \widehat{\lda}_0$ at infinity. In particular, we have an algebra isomorphism $\SH^*(V_0) \cong \SH^*(V_1)$.
\end{theorem}

\begin{remark}
   Symplectic cohomology is invariant under isomorphisms between convex symplectic manifolds; see \cite[Theorem 1.1]{BeRi20}.
\end{remark}


\section{Transfer maps}


\subsection{Transfer-admissible Hamiltonians}

Let $W$ be a complement-exact subdomain in a convex symplectic domain $(V, \ow, \lda)$, in the sense of Section \ref{sec: comp-exact}. Denote by $i\colon W \rightarrow \widehat V$ the obvious inclusion.

\begin{lemma}
The inclusion $i\colon W \rightarrow \widehat V$ extends to an embedding $i\colon \widehat W \rightarrow \widehat V$ with $i^* \widehat \lda_V = \widehat \lda_W$, where $\widehat \lda_W$ denotes the completed 1-form on $\widehat{W}$ from $\lda_W := \lda_V|_{W}$.
\end{lemma}

\begin{proof}
By the complement-exact condition, the primitive 1-form $\widehat \lda_V$ is well-defined on a region that includes the complement of $W$ and a neighborhood of $\p W$ in $W$, and $i^* \widehat \lda_V = \lda_W$. Let $X$ be the vector field $\widehat \ow_V$-dual to $\widehat \lda_V$, i.e., $\iota_X \widehat \ow_V = \widehat \lda_V$. Denote its flow by $\psi_X^t$. Then we define an extension $i\colon \widehat W \rightarrow \widehat V $ of the inclusion $i\colon W \rightarrow \widehat V$ by
$$
i(r_W, y) = \psi_X^{\log r_W} (y)
$$
where $(r_W, y) \in [1, \infty) \times \p W$. It is straightforward to see that $i^* \widehat \lda_V = \widehat \lda_W$.
\end{proof}

In the sense of the above lemma, we identify the completion $\widehat W$ with the image $i(\widehat W) \subset \widehat V$. In particular, the symplectization $[1, \infty) \times \p W$ can now be seen as a subset of the completion $\widehat V$. In the following, we denote the cylindrical coordinates of $\widehat W$ and $\widehat V$ by $r_W$ and $r_V$, respectively, and conventionally write $W$ for the region $\{r_W \leq 1\}$ in $\widehat V$.


Following \cite[Section 10.2]{Mclean_extension}, we shall define a certain class of admissible Hamiltonians $\{H = H(n)\}_{n \in \N}$ adapted to the embedding $i\colon \widehat{W} \hookrightarrow \widehat{V}$. Take a sufficiently small $\delta>0$ such that an inner collar neighborhood of $\p W \subset W$ can be written as $(1-2\delta, 1] \times \p W$ with the Liouville 1-form $r_W \alpha_W$.

\begin{lemma}\label{lem: radialcomparison}
On $i([1, \infty) \times \p W) \setminus V$, we have $r_V \leq r_W$. In particular, for any $A > 1$,
\[
\{r_W \leq A\} \subset \{r_V \leq A\}.
\]
\end{lemma}

\begin{proof}
Let $i(r_W, y) = \psi_X^{\log r_W}(y)$ lie outside $V$, and let $t_0 \in [0, \log r_W]$ be the first time at which the trajectory $\psi_X^s(y)$ reaches $\p V$. Then, we see that
\[
r_V = e^{\log r_W - t_0} \leq e^{\log r_W} = r_W.
\]
Since $W \subset V$, the stated inclusion follows.
\end{proof}

We now fix a positive constant $P > 0$, independent of the Hamiltonians below. By the lemma, $\{r_W \leq A\} \subset \{r_V < A+P\}$ for every $A > 1$.


\begin{definition}\label{transadmham}
A Hamiltonian $H\colon \widehat V \rightarrow \R$ is called \emph{transfer-admissible} if 
\begin{itemize}
\item $H$ is a $C^2$-small Morse function and $H \leq 0$  in $\{r_W < 1\}$;
\item $H$ is linear for $1-2\delta \leq r_W  \leq 1$ with a sufficiently small positive slope, less than the minimum of $\Spec(\p W, \alpha_W)$;
\item there exist constants $A > 1$ and $B >0$ such that $H - B$ is a $C^2$-small Morse function in the region $\{r_W \geq A\} \cap \{r_V \leq A+ P\}$;
\item $H$ depends only on $r_W$ in the region where $1-2\delta \leq r_W \leq A$, and it is linear for $1+ \epsilon/\kappa \leq r_W \leq A-\epsilon/\kappa$ with a positive slope $\kappa \not \in \Spec(\p W, \alpha_W)$, where $\epsilon >0$ is sufficiently small. On a neighborhood of $\{r_V \geq A+P\}$, $H$ depends only on $r_V$, and it is linear at infinity with a slope $\kappa/2 \not \in \Spec(\p V, \alpha_V)$;
\item $H$ is convex for $1 < r_W < 1 + \epsilon/\kappa$ and $A+P < r_V < A+P + \epsilon/\kappa$, and it is concave for $A - \epsilon/\kappa < r_W < A$.
\end{itemize}
\end{definition}
Since $H$ depends only on $r_W$ near $\p W$, we write $H(1)$ and $H'(1)$ for its value and radial derivative at $r_W = 1$, respectively.
See Figure \ref{fig: tranfer_Ham} for a conceptual description of transfer-admissible Hamiltonians. Note that transfer-admissible Hamiltonians are admissible in the sense of Section \ref{sec: admHam}. 

\begin{figure}[h]
    \centering
    \begin{overpic}[width=0.95\linewidth]{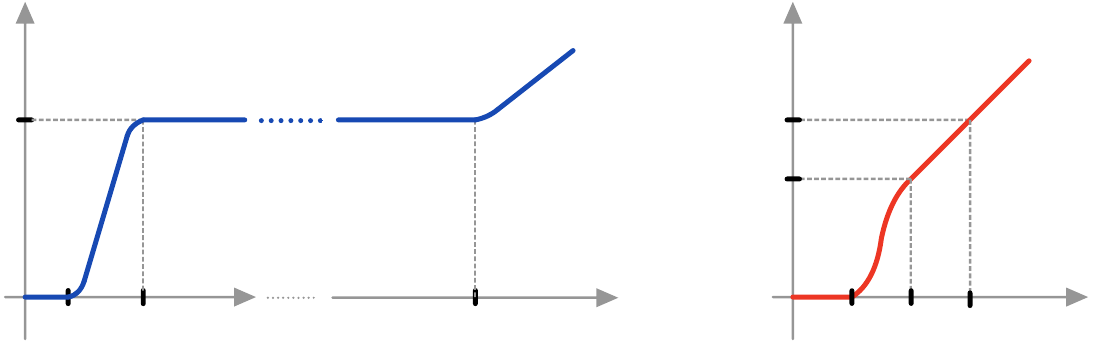}
    \put(90, 3){\scriptsize{$r_W$}}
    \put(-4, 83){\scriptsize{$B$}}
    \put(23, 3){\scriptsize{$1$}}
    \put(50, 3){\scriptsize{$A$}}
    \put(30, 50){$\kappa$} 
    \put(230, 3){\scriptsize{$r_V$}}
    \put(170, 3){\scriptsize{$A+P$}}
    \put(188, 101){$\frac{1}{2}\kappa$}
    \put(150, 70){\large{$H$}}

    \put(410, 3){\scriptsize{$r_W$}}
    \put(370, 3){\scriptsize{$1$}}
    \put(310, 3){\scriptsize{$1-2\delta$}}
    \put(340, 3){\scriptsize{$1-\delta$}}
    \put(380, 78){\large{$\phi$}}
    
    \end{overpic}
    \caption{A transfer-admissible Hamiltonian $H$ and a cutoff function $\phi$}
    \label{fig: tranfer_Ham}
\end{figure}

\begin{remark} \label{rem: ontheconstantB} Here are some remarks on Definition \ref{transadmham}.
\begin{enumerate}
   \item By the second property, there are no 1-periodic orbits of $H$ in the region $\{1-2\delta \leq r_W \leq 1\}$.
    \item By the conditions on $H$ for $r_W \leq A$, the difference $B - H(1)$ can be assumed to be arbitrarily close to $\kappa(A -1)$.
\end{enumerate}
\end{remark}
\subsection{Filtration by Reeb periods} \label{sec: newfiltration}
We define an action filtration adapted to the embedding $\widehat{W} \hookrightarrow \widehat{V}$, inspired by the construction in \cite[Appendix D]{McRi23}. Let $\phi\colon \widehat W \rightarrow \R$ be a smooth cutoff function such that
\[
\begin{cases}
    \phi =  0 & r_W \leq 1-2\delta, \\
     \phi'(r_W)\geq 0 & r_W \geq 1-2\delta,\\
    \phi(r_W) = r_W & r_W \geq 1-\delta.
\end{cases}
\]
See Figure \ref{fig: tranfer_Ham} for a description of $\phi$. Define a 1-form $\theta$ on $\widehat V$ by
$$
\theta = \begin{cases}  0 & r_W \leq 1-2\delta, \\ 
\phi(r_W) \alpha_W & 1-2\delta \leq r_W \leq 1, \\
\widehat \lda_V & \widehat{V}\setminus W.  \end{cases}
$$
\begin{remark}\label{rem: crit(H)inW} It is useful to observe the following.
\begin{enumerate}
   \item Since $\widehat \lda_V$ can be identified with $r_W \alpha_W = \phi(r_W) \alpha_W$ in a collar neighborhood of $\p W = \{r_W = 1\}$, the 1-form $\theta$ is well-defined in $\widehat{V}$, whereas $\phi$ is defined only in $\widehat{W}$. 
    \item In particular, every critical point in $W$ of a transfer-admissible Hamiltonian $H$ in Definition \ref{transadmham} lies in the region $\{r_W < 1-2\delta\}$ where $\theta$ vanishes.
\end{enumerate}
\end{remark}
For a transfer-admissible Hamiltonian $H$, define a function $f = f_H \colon \widehat V \rightarrow \R$ by
\begin{equation}\label{eq: def of f}
f = \begin{cases}  0 & r_W \leq 1-2\delta, \\
\displaystyle \int_{1-2\delta}^{r_W} \phi'(\tau) H'(\tau) d\tau & 1-2\delta \leq r_W \leq 1, \\
H -H(1) + H'(1) & \widehat{V}\setminus W.
\end{cases}
\end{equation}
Since $\phi(r_W) = r_W$ near $r_W = 1$ and $H$ is given as $H(r_W)=H'(1)r_W+H(1)-H'(1)$ in the inner collar of $\p W$, the two expressions for $f$ agree in a neighborhood of $r_W = 1$. 
Therefore, the pieces defining $f$ glue smoothly.

Now, we define an \emph{action functional} $\mathcal{F} =\mathcal{F}_{H}^{\theta} \colon \mathcal{L}\widehat V \rightarrow \R$ on the loop space $\mathcal{L}\widehat V$ by
$$
\mathcal{F}(x) = -\int_{S^1} x^* \theta + \int_{S^1} f(x(t)) dt.
$$

\begin{remark} \label{rem: exactandsignconv} 
When $V$ is globally exact, replacing $\theta$ and $f$ by a global primitive of $\widehat \ow_V$ and $H$, respectively, gives the usual action functional. This is a different choice from the cutoff construction used here.
\end{remark}

To show that $\mathcal{F}$ defines an action filtration on the cochain complex $\CF^*(H)$ of a transfer-admissible Hamiltonian $H$, we need to check that the action value does not increase along Floer trajectories:

\begin{proposition}
Let $u\colon \R \times S^1 \rightarrow \widehat V$ be a Floer trajectory from a 1-orbit $x_-$ to $x_+$. Then we have
$$
\mathcal{F}(x_+) - \mathcal{F}(x_-) \leq 0.
$$ 
\end{proposition}

\begin{proof}
Let $\eta = d \theta$. Note that 
$$
df = \begin{cases}
0 & r_W \leq 1-2\delta, \\
\phi'(r_W) H'(r_W)dr_W & 1-2\delta \leq r_W \leq 1, \\
dH & \widehat{V} \setminus W.
\end{cases}
$$
It follows that $\iota_{X_H} \eta = -df$. From a computation similar to that in \cite[Lemma D.1]{McRi23}, we find that $\eta(v, J v) \geq 0$ holds for each tangent vector $v \in T\widehat{V}$ where $J$ is cylindrical in a collar neighborhood of $\p W$ (i.e. $J \p_{r_W} = \frac{1}{r_W} R_{\alpha_W}$ and $J R_{\alpha_W} = -{r_W} \p_{r_W}$); here, the assumption on the cutoff function that $\phi, \phi' \geq 0$ is crucial.

We now compute, taking into account the Floer equation for $u$, that
$$
0 \leq \int_{\R \times S^1} \eta(\p_s u, J \p_s u) ds \wedge dt =  \int_{\R \times S^1} u^*\eta + u^*(\iota_{X_H}\eta) \wedge dt = \int_{\R \times S^1} d(u^*\theta) - d(f \circ u) \wedge dt.
$$
Observe that
$$
d((f \circ u) dt) = d(f \circ u) \wedge dt.
$$
It follows from Stokes' theorem that
\begin{align*}
0 \leq \int_{\R \times S^1} d(u^*\theta) - d(f \circ u) \wedge dt &= \int_{\R \times S^1} d(u^*\theta) - d((f \circ u) dt) \\
&= \int_{S^1} x_+^* \theta - f(x_+(t)) dt - \left(\int_{S^1} x_-^* \theta - f(x_-(t)) dt \right)\\  
&= -\mathcal{F}(x_+) + \mathcal{F}(x_-)
\end{align*}
which completes the proof.
\end{proof}
The action increasing property of $\mathcal{F}$ allows us to define the filtered cochain complex $\CF^*_{>a}(H)$, and hence the \emph{filtered Hamiltonian Floer cohomology} $\HF^*_{>a}(H)$, of a transfer-admissible Hamiltonian $H$, such that it is generated by 1-orbits $x$ over the Novikov field $\Lda_{V}$ whose action $\mathcal{F}(x) > a$. 

The compatibility of the filtration with continuation maps also follows from a similar computation. Let $H_{\pm}$ be transfer-admissible Hamiltonians with $H_+ \leq H_-$. Take a homotopy $H_s$ between $H_-$ and $H_+$, to define a continuation map $\HF^*(H_+) \rightarrow \HF^*(H_-)$, given by
\[
H_s = (1-\rho(s))H_- + \rho(s)H_+.
\]
Here, $\rho\colon \R \rightarrow [0,1]$ is a smooth monotone increasing function such that $\rho(s) = 0$ for $s \leq 0$ and $\rho(s) = 1$ for $s \geq 1$. From \eqref{eq: def of f}, this yields the corresponding smooth function $f_s$ satisfying
\[
f_s = (1-\rho(s))f_{H_-} + \rho(s)f_{H_+}.
\]
In particular, we may assume that $\p_s f_s = \rho'(s)(f_{H_+} - f_{H_-}) \leq 0$, passing to a cofinal family of Hamiltonians whose members agree on the fixed inner collar of $\p W$, so that $f_{H_+} \leq f_{H_-}$. Here cofinality is understood with respect to the slope at infinity, allowing additive constants as in \cite[Section 2.3]{Mclean_extension}. 
More generally, it suffices that $H_+'(1) \leq H_-'(1)$ and that $H_+ - H_+(1) \leq H_- - H_-(1)$ on $\widehat V \setminus W$, since $f_H = H'(1)\phi$ on the inner collar. With an $s$-family of almost complex structures $J = J_{s}$ chosen in a standard way, we consider a Floer trajectory $u\colon \R \times S^1 \rightarrow \widehat{V}$ of the parametrized Floer equation
\[
\p_s u + J_s(\p_t u - X_{H_s}) = 0
\]
from $x_- \in \mathcal{P}(H_-)$ to $x_+ \in \mathcal{P}(H_+)$. Then, as before, using $\iota_{X_{H_s}} \eta = -d_{\widehat V} f_s$, we have
\[
0 \leq \int_{\R \times S^1} \eta(\p_s u, J \p_s u) ds \wedge dt = \int_{\R \times S^1} d(u^*\theta) - d((f_s \circ u) dt) + \int_{\R \times S^1} (\p_s f_s) (u) ds \wedge dt.
\]
Here, the last term is due to the $s$-dependence of $f_s$; more precisely,
\[
d((f_s \circ u)dt) = u^*(d_{\widehat V} f_s) \wedge dt + (\p_s f_s)(u) ds \wedge dt.
\]
Stokes' theorem tells us that
\[
\mathcal{F}_{H_+}(x_+) - \mathcal{F}_{H_-}(x_-) \leq \int_{\R \times S^1} (\p_s f_s) (u) ds \wedge dt \leq 0
\]
as desired, where the second inequality follows from the condition that $\p_s f_s \leq 0$.

The resulting \emph{filtered symplectic cohomology} is denoted by $\SH^*_{>a}(V; \Lda_V)$. The filtered symplectic cohomology corresponding to the quotient complex
\[
\CF^*_{\leq a}(H) : = \CF^*(H) / \CF^*_{>a}(H)
\]
is denoted by $\SH^*_{\leq a} (V; \Lda_V)$. From the short exact sequence 
\[
0 \rightarrow \CF^*_{>a}(H) \rightarrow \CF^*(H) \rightarrow \CF^*_{\leq a}(H) \rightarrow 0
\]
we have the following long exact sequence.
\begin{corollary}\label{cor: actionles}
   There is a long exact sequence
    \[
    \rightarrow \SH^*_{> a}(V) \rightarrow \SH^*(V) \rightarrow \SH^*_{\leq a}(V) \rightarrow \SH^{*+1}_{> a}(V) \rightarrow
    \]
\end{corollary}



\subsection{Ring structure and filtration} To show that the product is well-defined on $\SH^*_{> a}(V)$, we observe that the $\mathcal{F}$-filtration satisfies the following inequality for a suitable choice of pair-of-pants Floer data; we refer the reader to \cite[Appendix A]{Rit} for more details. Let $H$ be a transfer-admissible Hamiltonian such that $2H$ is also transfer-admissible, and choose the pair-of-pants datum, asymptotic to $H_1 = H_2 = H$, $H_3 = 2H$, as follows: Take a closed $1$-form $\gamma$ on the surface $\mathcal{S}$ and a smooth function $\nu\colon \mathcal{S} \rightarrow [1, 2]$ such that
\[
\begin{cases}
    \gamma = dt, \; \nu = 1 & \text{near the positive punctures,}\\
    \gamma = 2dt, \; \nu = 2 & \text{near the negative puncture.} 
\end{cases}
\]
Set $\beta : = \frac{1}{\nu} \gamma$. Note that $\beta = dt$ near the punctures. As an $\mathcal{S}$-parametrized Hamiltonian $H_{\mathcal{S}}$, we take $H_{\mathcal{S}} : = \nu H$.

\begin{proposition}\label{prop: pantineq}
    With respect to the pair-of-pants datum chosen as above, let $u\colon \mathcal{S} \rightarrow \widehat V$ be a Floer solution of the equation \eqref{eq: Floereq_ring} converging to $x_1$, $x_2$ at the positive punctures and to $x_3$ at the negative puncture. Then we have
    \begin{equation}\label{eq: actioneqforpants}
       \mathcal{F}_{H_1}(x_1) + \mathcal{F}_{H_2}(x_2) \leq \mathcal{F}_{H_3}(x_3).
    \end{equation}
\end{proposition}

\begin{proof}

In view of \eqref{eq: def of f}, we obtain the corresponding $\mathcal{S}$-parametrized smooth function $f_{\mathcal{S}}$ from $H_{\mathcal{S}}$, and we see that
\[
f_{\mathcal{S}} = f_{\nu H} = \nu f_{H}.
\]
In particular, we have $d_{\mathcal{S}} (f_{\mathcal{S}}\beta) = d_{\mathcal{S}}(f_{H} \gamma) = f_H d_{\mathcal{S}}\gamma = 0$ as $\gamma$ is closed. Here, $d_{\mathcal{S}}$ denotes the derivative in the $\mathcal{S}$-direction. Also, note that
\[
X_{H_{\mathcal{S}}} \otimes \beta = X_{\nu H} \otimes \frac{\gamma}{\nu} = X_H \otimes \gamma.
\]

In local holomorphic coordinates $(s, t) \in \mathcal{S}$, write $\beta = \beta_s ds + \beta_t dt$ and
\[
Y_s := \p_s u - \beta_s X_{H_{\mathcal{S}}}, \quad Y_t : = \p_t u - \beta_t X_{H_{\mathcal{S}}}.
\]
Then the Floer equation \eqref{eq: Floereq_ring} implies $Y_t = J_{\mathcal{S}}Y_s$, and from the inequality $\eta(v, J v)\geq 0$, as in the case of Floer cylinders, we obtain 
\[
0 \leq \int_{\mathcal{S}} \eta(Y_s, J_{\mathcal{S}} Y_s) ds \wedge dt = \int_{\mathcal{S}} u^*\eta - u^*(d_{\widehat{V}} f_\mathcal{S}) \wedge \beta = \int_{\mathcal{S}} d(u^*\theta - (f_{\mathcal{S}} \circ u) \beta) + \int_{\mathcal{S}} d_{\mathcal{S}}(f_{\mathcal{S}} \beta)(u).
\]
By Stokes' theorem,
\[
-\mathcal{F}_{H_3}(x_3) + \mathcal{F}_{H_1}(x_1) + \mathcal{F}_{H_2}(x_2) \leq \int_{\mathcal{S}} d_{\mathcal{S}}(f_{\mathcal{S}} \beta) = 0
\]
which completes the proof.
\end{proof}

\begin{remark}
  The inequality \eqref{eq: actioneqforpants} implies that if the action values of the inputs are bounded from below by a non-negative number, that is, $a < \mathcal{F}(x_1)$ and $a < \mathcal{F}(x_2)$ for some $a \geq 0$, then so is the action value of the output; $a < \mathcal{F}(x_3)$. Therefore, it follows that the filtered symplectic cohomology $\SH^*_{> a}(V)$ admits the induced ring structure for $a \geq 0$.  
\end{remark}


\subsection{Filtered symplectic cohomology of non-positive actions} In this section, we show that the filtered symplectic cohomology $\SH^*_{\leq 0}(V; \Lda_V)$ can be canonically identified with the symplectic cohomology $\SH^*(W; \Lda_W)$ of the subdomain $W$.

\subsubsection{A cofinal family of transfer-admissible Hamiltonians}\label{sec: cofHam} Following \cite{Mclean_extension}, we choose a certain cofinal family of transfer-admissible Hamiltonians. For each $n \in \N$, we choose the constant numbers in Definition \ref{transadmham} as follows: 
\begin{itemize}
\item $0 < \kappa = \kappa(n) \rightarrow \infty$ as $n \rightarrow \infty$, and $\kappa \not \in \Spec(\p W, \alpha_W)$;
\item $\mu = \mu(n) : = \text{dist}(\kappa(n), \Spec(\p W, \alpha_W))$, and we choose $\kappa(n)$ so that $\mu(n) \rightarrow 0$ as $n \rightarrow \infty$;
\item $A = A(n) : = 6\kappa(n)/\mu(n)$, and we assume $A > \kappa > 1$;
\item $0< \epsilon = \epsilon(n) \rightarrow 0$ as $n \rightarrow \infty$.
\end{itemize}
Let $H = H(n)\colon \widehat V \rightarrow \R$ be a transfer-admissible Hamiltonian with the constants chosen as above. Then, the Hamiltonian 1-periodic orbits can only appear in the following regions; see Remarks  \ref{rem: ontheconstantB} and \ref{rem: crit(H)inW}.
\begin{enumerate}
\item[(I)] $r_W < 1-2\delta$;
\item[(II)] $1 < r_W < 1+ \epsilon/\kappa$ and $H''(r_W) > 0$;
\item[(III)]$A - \epsilon/\kappa < r_W < A$ and $H''(r_W) < 0$;
\item[(IV)] $A < r_W$ and $r_V < A+ P$; here, $H - H(1)$ is arbitrarily close to $\kappa(A -1)$;
\item[(V)] $A+P < r_V < A +P + \epsilon/\kappa$ and $H''(r_V) > 0$.
\end{enumerate}
See Figure \ref{fig: tranfer_Ham}. In each region, we estimate the $\mathcal{F}$-action value as follows:

\begin{lemma}\label{lem: actionorbits}
For sufficiently large $n \in \N$, every 1-orbit $x$ in the regions \emph{(I)} and \emph{(II)} has $\mathcal{F}(x) \leq 0$. In the other regions, we have $\mathcal{F}(x) > 0$.
\end{lemma}

\begin{proof}
We essentially mimic the proof in \cite[Lemma 10.2]{Mclean_extension} with additional care for the new action filtration $\mathcal{F}$. 


\vspace{0.5em}
\textbf{Case 1:} 1-orbits $x$ in the region (I) are critical points of $H$. By definition, $\theta = 0$ and $f = 0$ in this region. Therefore, $\mathcal{F}(x) = 0$.

\vspace{0.5em}
\textbf{Case 2:} Let $x$ be a 1-orbit in the region (II). Then $x$ lies in a hypersurface $\{r_x\} \times \p W \subset \widehat W$ for some $r_x \in (1, 1+\epsilon/\kappa)$. Since $X_H = H'(r_x) R_{\alpha_W}$ and $\theta = r_W \alpha_W$, it follows that
\begin{equation}\label{eq: validfrom2to5}
  \mathcal{F}(x) = - \int_{S^1} x^* \theta + \int_{S^1} f(r_x) dt = -r_xH'(r_x) + H(r_x) - H(1) + H'(1).  
\end{equation}
Observe that
\[
\int_{1}^{r_x} \tau H''(\tau) d\tau = \int_{1}^{r_x} \{(\tau H'(\tau))' - H'(\tau)\} d\tau = r_x H'(r_x) - H'(1) - H(r_x) + H(1).
\]
Therefore 
\[
\mathcal{F}(x) = - \int_{1}^{r_x} \tau H''(\tau) d\tau
\]
and since $H'' > 0$, we conclude $\mathcal{F}(x) < 0$.

\vspace{0.5em}
\textbf{Case 3:} Let $x$ be a 1-orbit in the region (III). Then $r_x \in (A -\epsilon/\kappa, A)$. Note that the equation \eqref{eq: validfrom2to5} still works in this region for the same reason. Since $H'(r_x) \in \Spec(\p W, \alpha_W)$, and $H'(r_x) < \kappa$ in this region, we have $H'(r_x) \leq \kappa - \mu$. Moreover, the difference $H(r_x) - H(1)$ is arbitrarily close to $\kappa(A-1)$; see Remark \ref{rem: ontheconstantB}. In other words, we can take a sequence $0 < e = e(n) \to 0$ such that 
\begin{equation}\label{eq: e(n)}
  |H(z) - H(1) - \kappa(A -1)|\leq e  
\end{equation}
for every point $z$ lying on a $1$-periodic orbit in regions (III), (IV), and (V).
It follows that
\begin{align}\label{eq: case_3_inequality}
\begin{split}
   \mathcal{F}(x) &= -r_xH'(r_x) + H(r_x) - H(1) + H'(1) \\
    &\geq - A(\kappa -\mu) + \kappa(A-1) -e + H'(1) \\
    &= \mu A - \kappa - e + H'(1) \\
    &= 5\kappa - e  + H'(1)
    \end{split}
\end{align}
where the last equality holds by the choice $A = 6\kappa/\mu$ so that $\mu A - \kappa = 5\kappa$. Since $\kappa = \kappa(n) \rightarrow \infty$ as $n \rightarrow \infty$, $\mathcal{F}(x) > 0$ for sufficiently large $n$.

\vspace{0.5em}
\textbf{Case 4:} Let $x$ be a 1-orbit in the region (IV). Then $x$ is a constant orbit. Hence
\[
\int_{S^1} x^*\theta = 0
\]
and, by the definition of $f$,
\[
\mathcal{F}(x) = f(x) = H(x) - H(1) + H'(1).
\]
By the choice of the error term $e = e(n)$, 
\[
H(x) - H(1) \geq \kappa(A-1) -e.
\]
Therefore, 
\begin{equation}\label{eq: case4_inequality}
    \mathcal{F}(x) \geq \kappa(A-1) - e + H'(1) > 0
\end{equation}

for sufficiently large $n$.

\vspace{0.5em}
\textbf{Case 5:} Let $x$ be a 1-orbit in the region (V). Then $r_x \in (A + P, A+ P+ \epsilon/\kappa)$. For the same reason as in the region (II), we know that
\[
\mathcal{F}(x) = -r_x H'(r_x) + H(r_x) - H(1) + H'(1).
\]
Observe that $H(r_x) - H(1)$ is still close to $\kappa(A - 1)$, and $0 \leq H'(r_x) \leq \frac{1}{2}\kappa$. 
Therefore, for sufficiently large $n$,
\begin{align*}
  \mathcal{F}(x) 
  &\geq -\frac{\kappa}{2}(A+P + \epsilon/\kappa) + \kappa(A-1) -e + H'(1)   \\
  &= \frac{\kappa}{2}(A - P -2) - \frac{\epsilon}{2} -e + H'(1).
\end{align*}
Here, note that $P$ is a fixed number independent of $n$. Since $A$ can be chosen arbitrarily large and $\epsilon$ arbitrarily small, we conclude that $\mathcal{F}(x) > 0$.
\end{proof}

\begin{theorem}\label{thm: nonnegandW}
Let $W$ be a complement-exact subdomain in a convex symplectic domain $V$ such that $\pi_1(\p W) = 0$ and $c_1(V)|_{\pi_2(V)} = 0$. Then
\begin{enumerate}
\item The inclusion $i\colon W \rightarrow V$ induces an isomorphism between the Novikov fields $\Lda_W$ and $\Lda_V$.
\item Under the above identification $\Lda_V \cong \Lda_W$, we have a canonical vector space isomorphism 
$$
\SH^*_{\leq 0}(V; \Lda_V) \cong \SH^*(W;\Lda_W).
$$
\end{enumerate}
\end{theorem}

\begin{proof}
The first assertion essentially follows from the topological assumption $\pi_1(\p W) = 0$ and the exactness of $\ow$ in the complement $\widehat{V} \setminus W$. It is enough to show that the induced map $i_*\colon \Gamma_W \rightarrow \Gamma_V$, via $i_*\colon \pi_2(W) \rightarrow \pi_2(V)$, is a group isomorphism. Since injectivity is obvious by definition, we prove the surjectivity. More precisely, we claim that for a class $[\gamma] \in \Gamma_V$, there is a class $\gamma_W \in \pi_2(W)$ such that $[i_*\gamma_W] = [\gamma] \in \Gamma_V$. If $\gamma(S^2) \subset W$, we put $\gamma_W := \gamma$. Suppose $\gamma(S^2) \not \subset W$. We may assume that the intersection $\gamma(S^2) \cap \p W$ consists of circles $\delta_1, \dots, \delta_{k}$ contained in $\p W$. Since $\pi_1(\p W) = 0$, each circle $\delta_i$ admits a capping disk $\nu_i$ in $\p W$. Now we consider the spheres in $W$ obtained from $\gamma$ by removing the parts of $\gamma$ in the complement $V \setminus W$ and attaching the capping disks $\nu_i$ along the boundary circles $\delta_i$. These spheres determine a class $\gamma_W \in \pi_2(W)$. 
Since $\ow$ is exact on the complement, we have $I_{\ow}(i_*(\gamma_W)) = I_{\ow}(\gamma)$, and hence $[i_*\gamma_W] = [\gamma] \in \Gamma_V$ as we claimed.


For the second assertion, let $H\colon \widehat V \rightarrow \R$ be a transfer-admissible Hamiltonian. We define an associated admissible Hamiltonian $H_W\colon \widehat W \rightarrow \R$ by setting $H_W := H|_{\widehat W}$ for $r_W \leq 1 + \epsilon/\kappa$ and taking $H_W$ to be linear for $r_W \geq 1 + \epsilon/\kappa$ with slope $\kappa$; see Figure \ref{fig: tranfer_Ham}. We accordingly define an associated admissible almost complex structure $J_W$ on $\widehat W$ for each admissible almost complex structure $J$ on $\widehat V$. 

Now we identify the chain group $\CF^*_{\leq 0}(H)$ with $\CF^*(H_W)$ as follows. By Lemma \ref{lem: actionorbits}, we know that $1$-periodic orbits $x$ of $H$ of non-positive $\mathcal{F}$-action are exactly those $1$-periodic orbits of $H_W$. Moreover, in the same way as in the proof of the first assertion identifying $\Lda_V$ with $\Lda_W$, we canonically obtain, for each capping disk $v$ of $x$, a capping disk $v_W$ of $x$ contained in $W$ such that $(x, v) = (x, v_W) \in \widetilde{\mathcal{L}_0 V}$. This yields an identification $\CF^*_{\leq 0}(H) = \CF^*(H_W)$. Moreover, the respective differential maps $d$ and $d_W$ can also be identified; the Floer trajectories $u\colon \R \times S^1 \rightarrow \widehat V$ between orbits in regions (I) and (II) have to be entirely contained in $\widehat W$ by the maximum principle \cite[Lemma 7.2]{AboSei}. Passing to the direct limits, increasing the slope $\kappa$ on both sides, we conclude that $\SH^*_{\leq 0}(V; \Lda_V) \cong \SH^*(W; \Lda_W)$.
\end{proof}

Note that the examples of negative line bundles in Example \ref{ex: neglinebdl2} satisfy all the assumptions of Theorem \ref{thm: nonnegandW}.

\subsubsection{Ring structure on \texorpdfstring{$\SH_{\leq 0}^*(V)$}{SH^*_<=0(V)}} \label{sec: ring on nonneg} Now we show that the product of the full symplectic cohomology $\SH^*(V)$ descends to the non-positive part $\SH_{\leq 0}^*(V)$; note that this does not directly follow from the inequality \eqref{eq: actioneqforpants}. The non-positive part $\SH_{\leq 0}^*(V)$ is defined by modding out the generators of $\SH^*(V)$ with positive actions at the chain level. It is therefore sufficient to show the following.

\begin{proposition}
    Let $u\colon \mathcal{S} \rightarrow \widehat V$ be a Floer solution of \eqref{eq: Floereq_ring} that converges to $x_1$, $x_2$ at positive punctures and to $x_3$ at the negative puncture. If $\mathcal{F}(x_1) > 0$ or $\mathcal{F}(x_2) > 0$, then $\mathcal{F}(x_3) > 0$. 
\end{proposition}

\begin{proof}
By the inequality \eqref{eq: actioneqforpants}, if the two inputs $x_1, x_2$ both have positive actions, then so does the output $x_3$. Now, without loss of generality, suppose that $\mathcal{F}(x_1) > 0$ and $\mathcal{F}(x_2) \leq 0$. Since we are interested in what happens after taking the direct limit over a cofinal family of transfer-admissible Hamiltonians, we may assume for simplicity that $x_1$ and $x_2$ are 1-orbits of a transfer-admissible Hamiltonian $H$ whose slope $\kappa$ is sufficiently large, and $x_3$ is a 1-orbit of $2 H$. 

In view of Lemma \ref{lem: actionorbits}, $x_2$ appears in the regions (I) or (II). In particular, if $x_2$ lies in the region (II), we see, as in the proof of Lemma \ref{lem: actionorbits}, that
\begin{align*}
 \mathcal{F}(x_2) &= -r_{x_2}H'(r_{x_2}) + (H(r_{x_2}) - H(1)) + H'(1) \\
 &\geq -(1+ \epsilon/\kappa) \kappa + 0 + H'(1) \\
 &= -\kappa - \epsilon + H'(1)\\
 &\geq -\kappa - \epsilon.
\end{align*}
Since $\mathcal{F}(x_2) = 0$ when $x_2$ lies in the region (I), we may assume that 
\[
\mathcal{F}(x_2) \geq -\kappa - \epsilon
\]
in both cases, taking $n \gg 1$.

On the other hand, Lemma \ref{lem: actionorbits} shows that $x_1$ lies in the region (III), (IV), or (V). Now we shall estimate the sum $\mathcal{F}(x_1) + \mathcal{F}(x_2)$ for each case to show that 
\[
\mathcal{F}(x_1) + \mathcal{F}(x_2) > 0
\]
for sufficiently large $n$. This implies that $\mathcal{F}(x_3) > 0$ by \eqref{eq: actioneqforpants}, as wanted.

Suppose $x_1$ lies in the region (III). Then, the same estimate as \eqref{eq: case_3_inequality} shows that 
\[
\mathcal{F}(x_1) \geq 5\kappa - e + H'(1)
\]
where $e = e(n) \to 0$ was chosen in \eqref{eq: e(n)}.
It follows that
\[
\mathcal{F}(x_1) + \mathcal{F}(x_2) \geq 4\kappa - \epsilon + H'(1) -e > 0
\]
for $n \gg 1$ since $\kappa \rightarrow \infty$ and $\epsilon \rightarrow 0$ as $n \rightarrow \infty$.

Suppose $x_1$ lies in the region (IV). Then it is a constant orbit, and as in \eqref{eq: case4_inequality},
\[
\mathcal{F}(x_1) \geq \kappa(A-1) - e +H'(1).
\]
Therefore, for $n \gg 1$,
\[
\mathcal{F}(x_1) + \mathcal{F}(x_2) \geq \kappa(A-2) - \epsilon + H'(1) -e > 0.
\]

Suppose $x_1$ lies in the region (V). Then the same computation as in Case 5 in the proof of Lemma \ref{lem: actionorbits} gives
\[
\mathcal{F}(x_1) \geq \frac{\kappa}{2}(A-P-2) -\frac{\epsilon}{2} + H'(1) -e.
\]
It follows that
\[
\mathcal{F}(x_1) + \mathcal{F}(x_2) \geq \frac{\kappa}{2}(A-P-4) - \frac{3\epsilon}{2} + H'(1) -e >0
\]
for $n \gg 1$. (Note that $P$ is a constant independent of $n$.) This completes the proof.
\end{proof}

As a corollary, we see that, with respect to the identification $\Lda := \Lda_V = \Lda_W$ as in Theorem \ref{thm: nonnegandW}, the non-positive part $\SH^*_{ \leq 0}(V)$ is canonically isomorphic to $\SH^*(W)$ as $\Lda$-algebras.
\begin{corollary}\label{cor: algebraisom}
    Under the assumptions of Theorem \ref{thm: nonnegandW}, we have a canonical $\Lda$-algebra isomorphism
    \[
    \SH^*_{ \leq 0}(V) \cong \SH^*(W).
    \]
\end{corollary}

\subsection{Transfer maps} Transfer-admissible Hamiltonians $H\colon \widehat V \rightarrow \R$ form a cofinal family of admissible Hamiltonians. The symplectic cohomology $\SH^*(V; \Lda_V)$ is the direct limit of Hamiltonian Floer cohomology $\HF^*(H)$ along the slope $\kappa$ of transfer-admissible Hamiltonians $H$. Consequently, the action functional $\mathcal{F}$ defines a natural map into the quotient
\begin{equation}\label{eq: fulltononpositive}
  \SH^*(V; \Lda_V) \rightarrow \SH_{\leq 0}^*(V; \Lda_V). 
\end{equation}
By the discussion in Section \ref{sec: ring on nonneg}, this map preserves the ring structure. Now, the composition with the identification $\SH^*_{\leq 0}(V; \Lda_V) \cong \SH^*(W; \Lda_W)$ by Corollary \ref{cor: algebraisom} yields an algebra homomorphism 
$$
\Phi \colon \SH^*(V; \Lda_V) \rightarrow \SH^*(W; \Lda_W)
$$
over $\Lda_V = \Lda_W$.
\begin{proof}[Proof of Theorem \ref{thm: A}]
It remains to establish the following commutative diagram.
\[
\begin{tikzcd}
        \SH^*(V) \arrow{r}{\Phi}  & \SH^*(W)  \\
        \QH^*(V) \arrow{r}{i^*} \arrow{u}{c^*}& \QH^*(W) \arrow[swap]{u}{c^*}
\end{tikzcd}
\]
For a transfer-admissible Hamiltonian $H^0$ with the slope $\kappa$ (and hence $\frac{1}{2}\kappa$) sufficiently small, 1-periodic orbits of $H^0$ are precisely the constant orbits corresponding to Morse critical points on $W$ and $V$ respectively. Moreover, the standard Floer theory, e.g. \cite{Fl88}, tells us that they recover the quantum cohomology $\QH^*(V)$ and $\QH^*(W)$ including the ring structure \cite[Lemma 13]{Ri14}. Now the restriction of the Hamiltonian $H^0\colon \widehat{V} \rightarrow \R$ to $H^0_W\colon \widehat{W}\rightarrow \R$, as in the proof of Theorem \ref{thm: nonnegandW}, corresponds to the natural restriction map $i^*\colon \QH^*(V) \rightarrow \QH^*(W)$. In other words, we have a commutative diagram:
\[
\begin{tikzcd}
        \SH^*(V) \arrow{r}  & \SH_{\leq 0}^*(V)   \\
        \QH^*(V) \arrow{r}{i^*} \arrow{u}{c^*}& \QH^*(W) \arrow[swap]{u}{c^*}
\end{tikzcd}
\]
Here, the upper horizontal map is the homomorphism in \eqref{eq: fulltononpositive}, and the canonical $c^*$-maps are defined by the inclusion; see \cite[Section 5]{Rit}. Now, applying the identification $\SH^*_{\leq 0}(V) \cong \SH^*(W)$ in Corollary \ref{cor: algebraisom}, we obtain the desired diagram.
\end{proof}


\section{Application}

\subsection{Subcritical handle attachment} \label{sec: subcrt_handle} Let $W$ be a convex symplectic domain of dimension $2n$. Denote the Weinstein handle of index $k \leq n$ by $\mathcal{H}^{2n}_k$. In particular, we say $\mathcal{H}^{2n}_k$ is \emph{subcritical} if $k < n$. We obtain another convex symplectic domain $V$ from $W$ by attaching the handle $\mathcal{H}^{2n}_k$ to the boundary $\p W$; we refer the reader to \cite[Section 2]{Fauck_handle} for details on handle attachment and the symplectic structure on $\mathcal{H}^{2n}_k$. Then $W$ is a complement-exact subdomain of $V$. If we further assume that $c_1(V)|_{\pi_2(V)} = 0$ and $\pi_1(\p W) = 0$, then the Novikov field $\Lda_V$ can be canonically identified with $\Lda_W$ via the inclusion $i\colon W \rightarrow V$ as in Theorem \ref{thm: nonnegandW}, and we obtain a transfer map $\Phi \colon \SH^*(V; \Lda_V) \rightarrow \SH^*(W; \Lda_W)$ by Theorem \ref{thm: A}. 

The following statement generalizes the seminal work of Cieliebak \cite{Ci02} on the invariance of symplectic cohomology for possibly non-exact convex symplectic domains.

\begin{theorem}\label{thm: inv_subcrit_hdl}
Suppose that $c_1(V)|_{\pi_2(V)} = 0$ and $\pi_1(\p W) = 0$. Under the identification $\Lda_V = \Lda_W$ as in Theorem \ref{thm: nonnegandW}, the transfer map 
\[
\Phi \colon \SH^*(V; \Lda_V) \rightarrow \SH^*(W; \Lda_W)
\]
yields an algebra isomorphism when $k < n$, i.e., the handle $\mathcal{H}^{2n}_k$ is subcritical.
\end{theorem}

\begin{proof}
The argument for the exact setup in the exposition of Fauck \cite{Fauck_handle} almost directly applies to the current, possibly non-exact setup. We provide a sketch of the proof based on \cite{Fauck_handle} while the role of the usual action filtration is replaced by the $\mathcal{F}$-filtration constructed in Section \ref{sec: newfiltration}.

In \cite[Sections 2.3 and 2.5]{Fauck_handle}, a cofinal family of transfer-admissible Hamiltonians $\{H_n \colon \widehat{V} \rightarrow \R \}_{n \in \N}$, in the sense of Section \ref{sec: cofHam}, is constructed with the following properties.
\begin{enumerate}
    \item For each $n$, we can take the handle $\mathcal{H}^{2n}_k$ sufficiently ``thin'' (see \cite[Section 2.2]{Fauck_handle}) so that every orbit of $H_n$ that intersects both $W$ and $\mathcal{H}^{2n}_k$ has a length greater than 1. 
    \item By the transfer-admissible condition on $H_n$ and Lemma \ref{lem: actionorbits}, every 1-orbit $x$ with non-positive action $\mathcal{F}(x) \leq 0$ is contained in $W$. 
    All remaining 1-orbits lie in the handle $\mathcal{H}^{2n}_k$ and have positive action.
    \item The Conley--Zehnder index $\CZ(x)$ of any 1-orbit $x$ lying in the handle $\mathcal{H}^{2n}_k$ turns out to be arbitrarily large as $n \rightarrow \infty$; see \cite[Section 2.4]{Fauck_handle}. This implies that the corresponding generators of the cochain complex $\CF^*_{> 0}(H_n)$ do not contribute to the symplectic cohomology for degree reasons.
\end{enumerate}
Now we consider the long exact sequence in Corollary \ref{cor: actionles}:
\[
    \rightarrow \SH^*_{> 0}(V) \rightarrow \SH^*(V) \rightarrow \SH^*_{\leq 0}(V) \rightarrow \SH^{*+1}_{> 0}(V) \rightarrow
    \]
The second property of the family $\{H_n\}_{n \in \N}$ yields a canonical isomorphism $\SH^*_{\leq 0}(V) \cong \SH^*(W)$ as stated in Corollary \ref{cor: algebraisom}. With this isomorphism, note that the middle map of the exact sequence $\SH^*(V) \rightarrow \SH^*_{\leq 0}(V) \cong \SH^*(W)$ coincides with the transfer map $\Phi \colon \SH^*(V) \rightarrow \SH^*(W)$ by definition. The second and third properties imply that 
\[
\SH^{*}_{> 0}(V) = 0.
\]
Therefore, by exactness, we conclude that the transfer map is an isomorphism, as asserted.
\end{proof}

\begin{remark}\label{rem: okaywith1hdl}
The assumption $c_1(V)|_{\pi_2(V)}=0$ is automatic for subcritical handles of index $k \geq 3$, provided that $c_1(W)|_{\pi_2(W)}=0$. Indeed, if
$k\ge 3$, then $\pi_2(V, W) = 0$. This implies that $i_* \colon \pi_2(W)\to\pi_2(V)$ induced by the inclusion is surjective, and $c_1(V)|_{\pi_2(V)}=0$ when $c_1(W)|_{\pi_2(W)}=0$.
\end{remark}

\subsection{Exotic convex symplectic structures} Inspired by  \cite{Mclean_extension}, we construct exotic convex symplectic structures as an application of transfer maps and the invariance in Theorem \ref{thm: inv_subcrit_hdl}, in terms of deformation equivalence in Definition \ref{def: deform_equiv}. 


\begin{theorem}\label{thm: exotic_str}
    Let $W$ be a convex symplectic domain of dimension $2n \geq 8$ with $c_1(W)|_{\pi_2(W)} = 0$ and $\pi_1(\p W) = 0$. Suppose that $\SH^0(W)$ is finite dimensional over $\Lda_W$. Then there exists an infinite family of convex symplectic domains $\{V_k\}_{k \in \N}$ such that $V_k$ is diffeomorphic to $W$, but $V_k$'s are pairwise deformation non-equivalent as convex symplectic domains.
\end{theorem}

\begin{proof}
    In \cite{Mclean_extension}, McLean constructed a family of finite-type Stein manifolds $\{E_k\}_{k \in \N}$ of dimension $2n \geq 8$ such that
    \begin{itemize}
    \item $E_k$ is diffeomorphic to $\C^n$;
    \item By the construction of $\{E_k\}_{k \in \N}$, $\SH^*(E_k; \Z_2) = \SH^*(E_1; \Z_2)^{k}$ as rings, and it is shown that $1 < i_{\Z_2}(E_1) < \infty$ where $i_{\Z_2}(E_1)$ denotes the number of idempotents of the ring $\SH^*(E_1; \Z_2)$. In particular, $\SH^*(E_k; \Z_2) \not \cong \SH^*(E_{\ell}; \Z_2)$ as rings for $k \not = \ell$. 
    \end{itemize}
Since $E_k$ is of finite type, we can choose a subdomain $D_k \subset E_k$ such that $D_k$ is diffeomorphic to a ball $B^{2n}$, and the completion $\widehat{D_k}$ is Stein deformation equivalent to $E_k$.

For each $k \in \N$, we take the boundary connected sum of $W$ with $D_k$, and denote the resulting convex symplectic domain by $V_k$. In other words, $V_k$ is obtained by attaching the Weinstein 1-handle $\mathcal{H}^{2n}_1$ to the boundary of the disjoint union $W \sqcup D_k$. Note that $V_k$ is diffeomorphic to $W$ since $D_k$ is nothing but a ball, and $c_1(V_k)|_{\pi_2(V_k)} = 0$ as we assumed $c_1(W)|_{\pi_2(W)} = 0$. In addition, $W$ is complement-exact in $V_k$, so in view of Theorem \ref{thm: nonnegandW}, we can identify the Novikov field $\Lambda_{V_k}$ with $\Lambda_W$.

Now, we consider the complement-exact subdomain
\[
W\sqcup D_k\subset V_k.
\]
Although $W\sqcup D_k$ is disconnected, each boundary component is simply-connected, and the arguments of Theorem \ref{thm: nonnegandW} and Corollary \ref{cor: algebraisom} apply component-wise over the common coefficient field $\Lambda_W$. In particular, the non-positive part $\SH^*_{\le 0}(V_k;\Lda_{V_k})$ splits as
\[
\SH^*_{\le 0}(V_k;\Lda_{V_k})
\cong
\SH^*(W;\Lda_W)\times \SH^*(D_k;\Lda_W)
\]
as algebras under the identification $\Lambda_{V_k}\cong \Lambda_W$. Here,
\[
\SH^*(D_k;\Lda_W) := \SH^*(D_k;\Z_2)\otimes_{\Z_2}\Lda_W
\]
where the intrinsic Novikov field of $D_k$ is $\mathbb Z_2$. Since the handle $\mathcal{H}_1^{2n}$ is subcritical, the argument as in Theorem \ref{thm: inv_subcrit_hdl} shows that $\SH^*_{>0}(V_k)=0.$ Again, it follows from the long exact sequence in Corollary \ref{cor: actionles} that
\[
\SH^*(V_k)
\cong
\SH^*_{\le0}(V_k).
\]
Consequently,
\begin{equation}\label{eq: V_k_isom_W_and_D_k}
 \SH^*(V_k;\Lambda_{V_k})
\cong
\SH^*(W;\Lambda_W)\times
\SH^*(D_k;\Lambda_W)
\end{equation}
as algebras under the identification $\Lambda_{V_k}\cong \Lambda_W$.

Since $\SH^0(W)$ is finite dimensional over $\Lda_W$ by assumption, the number of idempotents $i_{\Lda_W}(W)$ of $\SH^*(W)$ over $\Lda_W$ is finite by Lemma \ref{lem: alg_finite_idempotents} below. In McLean \cite[Section 7.3]{Mclean_extension}, it is proved that the number of idempotents $i_{\Z_2}(D_1)$ over $\Z_2$ is greater than one and finite, and the argument of McLean still works over $\Lda_W$ so that $1 < i_{\Lda_W}(D_1) < \infty$; it is shown that a certain subspace containing all idempotents is finite dimensional over a field of characteristic two, and this observation is independent of the choice of coefficient field of characteristic two. Note that
\begin{align*}
   \SH^*(D_k; \Lda_W) &= \SH^*(D_k; \Z_2) \otimes_{\Z_2} \Lda_W  \\
   &= (\SH^*(D_1; \Z_2)^{k}) \otimes_{\Z_2} \Lda_W\\
   &= (\SH^*(D_1; \Z_2) \otimes_{\Z_2} \Lda_W)^{k}.
\end{align*}
It follows that $i_{\Lda_W}(D_k) = (i_{\Lda_W}(D_1))^k$, so $i_{\Lda_W}(D_k) \neq i_{\Lda_W}(D_{\ell})$ if $k \neq \ell$.

By the isomorphism \eqref{eq: V_k_isom_W_and_D_k},
\[
i_{\Lda_{V_k}}(V_k) = i_{\Lda_W}(W) \cdot i_{\Lda_W}(D_k).
\]
It follows that $i_{\Lda_{V_k}}(V_k) \not = i_{\Lda_{V_{\ell}}}(V_{\ell})$ for $k \not = \ell$. This completes the proof in view of the invariance of symplectic cohomology in Theorem \ref{thm: inv_SH_BR}.
\end{proof}

\begin{lemma}\label{lem: alg_finite_idempotents}
Let $A = \bigoplus_{k \in \Z} A_k$ be a $\Z$-graded commutative $\Lda$-algebra where $\Lda$ is a field of characteristic 2. Assume that the vector space $A_0$ is finite dimensional. Then $A$ has only finitely many idempotents.
\end{lemma}

\begin{proof}
Denote the set of idempotents of $A$ by $I$. Since $A$ is commutative of characteristic 2 and the ring structure has no degree shift, we have $I \subset A_0$. Hence $\text{Span}(I) \subset A_0$ is finite dimensional. Moreover, since $I$ spans the finite-dimensional space $\text{Span}(I)$, we can choose a basis $\mathcal{B} = \{b_1, b_2, \dots, b_n\} \subset I$.

Now, take $a \in I$, i.e. $a^2 = a$. Then, $a  = \sum_{j=1}^n c_j b_j$ for some $c_j \in \Lda$. Using commutativity and the fact that $\Lda$ is of characteristic 2, we find $a^2 = \sum_{j=1}^n c_j^2 b_j$, and since $a^2 = a$, we further have
\[
\sum_{j=1}^n c_j^2 b_j =  a^2  = a = \sum_{j=1}^n c_j b_j.
\]
Since $\mathcal{B}$ is a basis, this implies that $c_j^2 = c_j$. As $\Lda$ is a field, it follows that $c_j = 0, 1$. From the finiteness of $\mathcal{B}$, we see that there are only finitely many possible linear combinations $\sum_j c_j b_j$ with $c_j = 0, 1$. Therefore, $I$ is a finite set.
\end{proof}

\begin{example}
As a concrete example, for $n \geq 3$, the disk bundle $D(K_B)$ over the smooth hypersurface $B = X_{n+1} \subset \CP^{n+1}$ of degree $n+1$, introduced in Example \ref{ex: neglinebdl2}, satisfies all the assumptions of Theorem \ref{thm: exotic_str}. In particular, by Ritter \cite[Theorem 6]{Ri14}, the symplectic cohomology $\SH^*(D(K_B))$ is trivial. It follows that $D(K_B)$, which is a non-exact convex symplectic domain, admits infinitely many convex symplectic structures that are pairwise deformation non-equivalent. 
\end{example}
 
\subsection*{Acknowledgement} The authors would like to thank Otto van Koert for helpful discussions and the anonymous referee for valuable comments. The first author was supported by the National Research Foundation of Korea(NRF) grant funded by the Korea government(MSIT) (No. RS-2025-23524132, NRF-2021R1F1A1060118). The second author was supported by JSPS KAKENHI Grant Numbers JP22K13913, JP24H00182, JP26K06780.

\bibliographystyle{abbrv}
\bibliography{mybibfile}

\end{document}